\documentclass{amsart}
\usepackage{indentfirst}
\usepackage{url}
\usepackage{enumerate}
\usepackage{amsthm}
\usepackage{amsmath,mathrsfs}
\usepackage{comment}
\usepackage{amssymb}
\usepackage{booktabs}

\newtheorem{thm}{Theorem}

\newtheorem{Lem}{Lemma}
\newtheorem{remark}{Remark}
\newtheorem{cor}{Corollary}

\newlength{\bibitemsep}
\newlength{\bibparskip}
\let\oldthebibliography\thebibliography
\renewcommand\thebibliography[1]{%
  \oldthebibliography{#1}%
  \setlength{\parskip}{\bibitemsep}%
  \setlength{\itemsep}{\bibparskip}%
}

\usepackage[dvipsnames]{xcolor}
\usepackage{todonotes}
\author{}
\date{}
\begin{document}
\title[The error term of the fourth moment of the zeta-function]{The error term of the fourth moment of the Riemann zeta-function}
\author[B. Durkan]{Benjamin Durkan}
\address{Department of Mathematics, University of Manchester\\
Manchester, M13 9PL, United Kingdom}
\email{benjamin.durkan@postgrad.manchester.ac.uk}

\author[N. Paloj\"{a}rvi]{Neea Paloj\"{a}rvi}\thanks{Supported by the Finnish Cultural Foundation.}
\address{School of Science\\
UNSW Canberra at ADFA\\
ACT, Australia}
\email{n.palojarvi@unsw.edu.au}

\author[T.~S. Trudgian]{Timothy S. Trudgian}\thanks{This work was supported by Australian Research Council Discovery Project DP240100186.}
\address{School of Science\\
UNSW Canberra at ADFA\\
ACT, Australia}
\email{t.trudgian@adfa.edu.au}


\begin{abstract}
\noindent
We examine the size of $E_{2}(T)$, the error term in the asymptotic formula for $\int_{0}^{T} |\zeta(1/2 + it)|^{4}\, dt$ where $\zeta(s)$ is the Riemann zeta-function. We reduce the powers of $\log T$ in the known bounds for $E_{2}(T)$ and $\int_{0}^{T} E_{2}(t)^{2}\, dt$. As a consequence, we obtain small logarithmic improvements for $k$th moments where $8\leq k\leq 12$. 
In particular, we obtain $\int_0^T|\zeta(1/2+it)|^{12}\,dt\ll T^2(\log T)^{15}$.
\end{abstract}
\maketitle

\section{Introduction}\label{leubald}
\noindent
Consider the fourth moment of the Riemann zeta-function
\begin{equation}\label{als}
\int_{0}^{T} |\zeta(1/2 + it)|^{4} \, dt = T P_{4}(\log T) + E_{2}(T),
\end{equation}
where $P_{4}(x)$ is a quartic polynomial, the coefficients of which are known. In this paper we consider three estimates involving $E_{2}(T)$.
The first is the size of $E_{2}(T)$ itself. Heath-Brown \cite{RHB4}, proved (\ref{als}) and showed that $E_{2}(T) \ll T^{7/8 + \epsilon}$. This was improved by Ivi\'{c} and Motohashi \cite{Ivic91,IvicM90,IvicM95} to
\begin{equation}\label{fraulenist}
E_{2}(T) \ll T^{2/3} (\log T)^{C_{1}},
\end{equation}
for some constant $C_{1}$, with further analysis conducted by Ivi\'{c} in \cite{Ivic99,Ivic00}.
The second is the mean-square of $E_{2}(T)$, which was also studied by Ivi\'{c} and Motohashi, who showed in \cite{IM94} that
\begin{equation}\label{laune}
\int_{0}^{T} E_{2}(t)^{2} \, dt\ll T^{2} (\log T)^{C_{2}},
\end{equation}
for some constant $C_{2}$.
The third is the 12th power moment estimate for $\zeta(s)$. Heath-Brown \cite{RHB12} proved that
\begin{equation}\label{hochzeit}
\int_{0}^{T} |\zeta(1/2 + it)|^{12}\, dt \ll T^{2} (\log T)^{17}.
\end{equation}
At first glance, Heath-Brown's result does not involve $E_{2}(T)$, and, indeed, his proof involves Kloosterman sums, the mean-square of $\zeta(s)$ in short intervals, and (\ref{als}). More general methods can be used: in 2007, Sankaranarayanan \cite{Sank} proved a version of (\ref{hochzeit}) with $33$ instead of $17$.

Iwaniec \cite{Iwaniec} was the first to prove an appropriate short-interval version of (\ref{als}) and to apply this to the 12th power estimate --- see also \cite{IvicB,IvicHardy}. This leads to $\int_{0}^{T} |\zeta(1/2 + it)|^{12}\, dt \ll T^{2 + \epsilon}$. Ivi\'{c} and Motohashi \cite{IM94,IvicM95} were able to refine some aspects of Iwaniec's method to improve this to an estimate of the form
\begin{equation*}\label{feen}
\int_{0}^{T} |\zeta(1/2 + it)|^{12}\, dt \ll T^{2} (\log T)^{8+C_{2}}, 
\end{equation*}
with $C_{2}$ coming from (\ref{laune}). The proof of Corollary \ref{cor:12moment} below saves one logarithm in this passage, giving $7+C_{2}$ instead of $8+C_{2}$.
This leads us to investigate the size of $C_{2}$ in (\ref{laune}).

Motohashi \cite{MotohashiB} proved that one can take $C_{2} = 22$, and, via similar methods, $C_{1} = 8$ in (\ref{fraulenist}). In this article we show that we may take $C_{1} = 10/3$ and $C_{2} = 8$. In fact, we show slightly more in our main results below.

\begin{thm}
\label{thm:fourthMoment}
    We have
    \begin{equation*}
        E_{2}(T) \ll T^{2/3} \left(\log{T}\right)^{10/3}, \quad\text{as } T\to \infty.
    \end{equation*}
\end{thm}

\begin{thm}\label{die}
Let $A \in[0,2]$ be fixed. Then, we have
\begin{equation*}
 \int_{0}^{T} |E_{2}(t)|^{A}\, dt \ll_A
 T^{1+A/2}(\log{T})^{4A},\quad\text{as } T\to \infty.
\end{equation*}
\end{thm}

Theorem \ref{die} gives us an estimate for the $k$th moment\footnote{While updating this work, we note that de Faveri and Pandey \cite{FP} have announced an improved bound for  $\int_{0}^{T} |\zeta(1/2 + it)|^{k}\, dt$ for all $4<k\leq 12$.} when $8\leq k \leq 12$.
\begin{cor}
\label{cor:12moment}
Let $k \in [8,12]$ be fixed. Then
\begin{equation*}
\int_{0}^{T} |\zeta(1/2 + it)|^{k}\, dt \ll_k T^{1+\frac{k-4}{8}} (\log{T})^{\frac{11k-12}{8}},
\quad\text{as } T\to \infty.
\end{equation*}
\end{cor}

\begin{remark}\label{cups}
    If $k=12$, then Corollary \ref{cor:12moment} gives
    \begin{equation*}
        \int_{0}^{T} |\zeta(1/2 + it)|^{12}\, dt \ll T^{2} (\log{T})^{15},
    \end{equation*}
    which improves by two logarithms Heath-Brown's estimate \eqref{hochzeit} using a different method.  The bound provided in \cite[Theorem 7.2]{IvicB1983} for $k \in [8,12]$ is $\ll T^{1+(k-4)/8+\varepsilon}$ for any $\varepsilon>0$. We reprove this, and are able to keep track of the logarithmic powers.
\end{remark}

\begin{remark}
    Following the same structure as in the proofs of Theorem \ref{thm:fourthMoment} and Corollary \ref{cor:12moment}, we can prove that $\int_{0}^{T} |\zeta(1/2 + it)|^{k}\, dt \ll T^{k/6+\varepsilon}$ for all $k >12$ and any $\varepsilon>0$. However, as in the proof of Theorem 7.2 in \cite{IvicB1983}, we can state something stronger for $k \geq 178/13 \approx 13.7$, namely, that for any $\varepsilon>0$
    \begin{equation*}
    \label{eq:highMoments}
        \int_{0}^{T} |\zeta(1/2 + it)|^{k}\, dt \ll T^{1+\left(k-\frac{178}{13}\right)c+\frac{29}{13}+\varepsilon}, \quad \text{as } T \to \infty,
    \end{equation*}
    where $\zeta(1/2+i t)\ll T^{c+\varepsilon}$ when $|t|\leq T$. By  \cite[Theorem 5]{Bourgain} we can choose $c=13/84$. 
\end{remark}

Central to the so-called spectral-theoretic approach are moments of the Hecke series $H_{j}(s)$ at the central point $s=1/2$. Indeed, 
let $\{\lambda_j=\kappa_j^2+1/4, \kappa_j>0\} \cup \{0\}$ be the discrete spectrum of the hyperbolic Laplacian acting on $SL(2,\mathbb{Z})\setminus \mathbb{H}$. Moreover,
let $\alpha_j:=|\rho_j(1)|^2(\cosh{(\pi \kappa_j)})^{-1}$, where $\rho_j(1)$ is the first Fourier coefficient of the Maass form corresponding to $\lambda_j$ to which the $L$-function $H_j(s)$ is attached. We note that $H_j(1/2) \geq 0$ by \cite{KatokSarnak}, and from the theory of the Selberg zeta-function (see e.g.\ \cite[(5.49)]{Ivic91}) it follows that $|\{\kappa_j: \kappa_j \leq T\}|=T^2/12+O(T\log{T})$. For a more detailed introduction to Hecke series, see \cite{Ivic91}.

One requires upper bounds of the form
\begin{equation}\label{gluckliche}
\sum_{\kappa_{j}\leq K} \alpha_{j} H_{j}^{3}(1/2) \ll K^{2} \log^{A_{3}} K, \quad  \sum_{\kappa_{j}\leq K} \alpha_{j} H_{j}^{4}(1/2) \ll K^{2} \log^{A_{4}} K.
\end{equation}
Motohashi's method (see also \cite{Ivic14}) gives $C_{1} \leq (11 + 4A_{3})/6$. Indeed, Motohashi showed $A_{3} \leq 8$, whence  $C_{1}\leq 43/6 < 8$. Ivi\'c \cite{Ivic02} gave asymptotic expressions\footnote{We note that these have been refined with a power-savings error term in, e.g., \cite{Qi}.} for the sums in (\ref{gluckliche}), whence one can take $A_{3} = 3$ and $A_{4} = 6$. We use these improvements, and other refinements, to prove our results in the following sections.

\section{Proof of Theorem \ref{thm:fourthMoment}}
In this section, we prove Theorem \ref{thm:fourthMoment} by estimating $E_2(T)$ with the function
\begin{equation*}
    I(T,\Delta):= \frac{1}{\Delta \sqrt{\pi}} \int_{-\infty}^\infty \left|\zeta\left(1/2+i(T+t)\right)\right|^4e^{-(t/\Delta)^2} \, dt, \qquad (0<\Delta<T/\log{T}).
\end{equation*}
The following lemma describes the correspondence between $I(T, \Delta)$ and the fourth moment of the Riemann zeta-function. We include it so that we can easily show our (\ref{eq:ErrorBound}) below, which improves on the corresponding \cite[(94)]{Ivic14}. We then refine the comparison by retaining the boundary integrals.

\begin{Lem}
\label{lemma:I4}
Let $\eta \in [1/2,1)$, $T^\eta\leq \Delta \leq T\exp(-\sqrt{\log{T}})$, $T_1:=T-\Delta(\log{T})^{1/2}$, $T_2:=2T+\Delta(\log{T})^{1/2}$, $T_3:=T+2\Delta(\log{T})^{1/2}$ and $T_4:=2T-2\Delta(\log{T})^{1/2}$.
Then, as $T\to\infty$,
\begin{equation*}
    \int_{T_3}^{T_4} I(t,\Delta) \, dt+O(\log^4{T}) \leq \int_{T}^{2T}\left|\zeta(1/2+ix)\right|^4 \, dx \leq \int_{T_1}^{T_2} I(t,\Delta) \, dt+O(\log^4{T}).
\end{equation*}
\end{Lem}

\begin{proof}
Put $F(x)=|\zeta(1/2+ix)|^4$ and
$W_{p,q}(x)=(\Delta\sqrt\pi)^{-1}\int_p^q e^{-(t-x)^2/\Delta^2}\,dt$.
By Fubini's theorem, $\int_p^q I(t,\Delta)\,dt=\int_{\mathbb R}F(x)W_{p,q}(x)\,dx$.
For $T\leq x\leq2T$, the Gaussian tails give
\begin{equation*}\label{eq:eBetween}
 1-W_{T_1,T_2}(x)
 \ll\int_{\sqrt{\log T}}^\infty e^{-v^2}\,dv\ll T^{-1}.
\end{equation*}
Positivity and \eqref{als} therefore prove the upper bound, with error
$T^{-1}\int_T^{2T}F(x)\,dx\ll(\log T)^4$.

For the lower bound, $0\leq W_{T_3,T_4}(x)\leq1$ on $[T,2T]$.
At $x=T-u$ or $x=2T+u$, with $u\geq0$, the distance to
$[T_3,T_4]$ is at least $u+2\Delta\sqrt{\log T}$. Hence
$W_{T_3,T_4}(x)\ll T^{-4}e^{-(u/\Delta)^2}$.
Using $F(x)\ll1+|x|^{2/3}$, we obtain
\begin{equation*}\label{eq:upperGeneral}
\begin{split}
 \int_{\mathbb R\setminus[T,2T]}F(x)W_{T_3,T_4}(x)\,dx
 &\ll T^{-4}\int_0^\infty
 (1+T^{2/3}+u^{2/3})e^{-(u/\Delta)^2}\,du\\
 &\ll T^{-4}\bigl(\Delta(1+T^{2/3})+\Delta^{5/3}\bigr)\ll1.
\end{split}
\end{equation*}
Fubini and positivity now give the lower bound.
\end{proof}


The following refinement avoids the logarithmic shifts of the endpoints.

\begin{Lem}\label{lemma:I4Boundary}
Let $a<b$ and $0<2\Delta\leq b-a$. Then
\begin{equation*}\label{eq:GaussianComparison}
\begin{split}
\left|\int_a^b |\zeta(1/2+ix)|^4\,dx-\int_a^b I(t,\Delta)\,dt\right|
\leq 2\int_a^{a+\Delta}I(t,\Delta)\,dt
+2\int_{b-\Delta}^{b}I(t,\Delta)\,dt.
\end{split}
\end{equation*}
\end{Lem}

\begin{proof}
We first note that, for $s\geq0$,
\begin{equation}\label{eq:GaussianTails}
\int_s^\infty e^{-u^2}\,du\leq2\int_s^{s+1}e^{-u^2}\,du,
\qquad
\int_s^\infty e^{-u^2}\,du\leq2\int_{s-1}^{s}e^{-u^2}\,du.
\end{equation}
For the first inequality, split the tail into intervals of length one and use $e^{-(u+n)^2}\leq e^{-n^2}e^{-u^2}$ for $u\geq s$ and $\sum_{n\geq0}e^{-n^2}<2$. The second follows from the first and monotonicity if $s\geq1$. If $0\leq s\leq1$, the integral over $[s-1,s]$ is smallest at an endpoint. Moreover, $\int_0^1e^{-u^2}\,du\geq\int_0^1(1-u^2)\,du=2/3$, whereas $\int_s^\infty e^{-u^2}\,du\leq\sqrt{\pi}/2$.

Put
\begin{equation*}
G(x):=\frac{1}{\Delta\sqrt{\pi}}\int_a^b e^{-(t-x)^2/\Delta^2}\,dt.
\end{equation*}
After translation and division by $\Delta$, the first inequality in \eqref{eq:GaussianTails} bounds $G(x)$ when $x$ is outside $[a,b]$. The second bounds the two missing tails when $x\in [a,b]$. Thus
\begin{equation*}
\left|\boldsymbol{1}_{[a,b]}(x)-G(x)\right|
\leq \frac{2}{\Delta\sqrt{\pi}}
\left(\int_a^{a+\Delta}+\int_{b-\Delta}^{b}\right)e^{-(t-x)^2/\Delta^2}\,dt.
\end{equation*}
Multiplication by $|\zeta(1/2+ix)|^4$, integration over $x$, and Fubini prove the claim.
\end{proof}

\begin{proof}[Proof of Theorem \ref{thm:fourthMoment}]
We follow \cite[Section 5.2]{MotohashiB} and \cite{Ivic14}, using the
boundary comparison in Lemma \ref{lemma:I4Boundary}. Define
\begin{equation}\label{eq:defS}
 S(T,\Delta):=\pi\sqrt{\frac T2}\sum_{j=1}^\infty
 \alpha_j\kappa_j^{-3/2}H_j^3(1/2)
 \cos\left(\kappa_j\log\frac{\kappa_j}{4eT}\right)
 e^{-(\Delta\kappa_j/(2T))^2}.
\end{equation}
For comparison, \eqref{als}, Lemma \ref{lemma:I4} and
\cite[(92), (93)]{Ivic14} give
%
\begin{equation}\label{eq:ErrorBound}
 E_2(2T)-E_2(T)\leq S(T_2,\Delta)-S(T_1,\Delta)+O\bigl(\Delta\log^{4.5}T+T^{1/2}\log^BT\bigr),
\end{equation}
with a corresponding lower bound using $S(T_4,\Delta)-S(T_3,\Delta)$.
Here $T_1,\ldots,T_4$ and the range of $\Delta$ are as in Lemma \ref{lemma:I4}.
For the sharper estimate, \cite[(4.8), (4.11), (4.12)]{Ivic14} give
\begin{equation}\label{eq:integratedI}
\int_0^T I(t,\Delta)\,dt
=TP_4(\log T)+S(T,\Delta)+O_D\left(T^{1/2}\log^B T\right),
\end{equation}
provided that, for a fixed $D>0$, we have $\Delta \in[T^{1/2}(\log T)^{-D}, T\exp(-\sqrt{\log T})].$
Here $B$ is fixed once $D$ is fixed. Apply Lemma \ref{lemma:I4Boundary}
with $a=T$, $b=2T$, and evaluate its integrals by differences of
\eqref{eq:integratedI}. The polynomial contributes $O(\Delta(\log T)^4)$
on each boundary interval. Together with \eqref{als}, this gives
\begin{equation}\label{eq:ErrorBoundary}
|E_2(2T)-E_2(T)|
\ll\Delta(\log T)^4+\sup_{x\asymp T}|S(x,\Delta)|+T^{1/2}(\log T)^B.
\end{equation}
Next, we will estimate the term $S(T, \Delta)$.

Put $Q=T\Delta^{-1}$. Let $A>0$ be sufficiently large. As in \cite[p. 37]{Ivic14}, the contribution from the part where $\kappa_j >AQ\sqrt{\log{T}}$ is $O(1)$. We divide the rest of the sum into three parts at the points $\kappa_j = Q$ and $\kappa_j= AQ(\log\log{T})^{1/2}$. By partial summation and the first estimate in \eqref{gluckliche}, in the first case we have
\begin{align}
    &\sum_{\kappa_j \leq Q} \alpha_j\kappa_j^{-3/2}H_j^3(1/2)\cos\left(\kappa_j\log{\frac{\kappa_j}{4eT}}\right)e^{-\left(\frac{\Delta \kappa_j}{2T}\right)^2} \ll \left(Q\right)^{1/2}\left(\log{(Q)}\right)^{3} \label{eq:SFirst}.
\end{align}
Similarly, setting $v=\Delta u/(2T)$ in the integral, the second sum is $\ll$
\begin{equation}
\begin{split}
\frac{\left(AQ(\log\log{T})^{1/2}\right)^{1/2}\left(\log{(AQ)}\right)^3}{(\log{T})^{A^2/4}} &+\left(Q\right)^{1/2}\left(\log{(Q)}\right)^{3} \\
    &+Q^{-2}\int_{Q}^{AQ(\log\log{T})^{1/2}} \frac{u^{3/2}(\log{u})^{3}}{e^{\left(\frac{\Delta u}{2T}\right)^2}} \, du  \\
    &\ll \left(AQ\right)^{1/2}(\log{(AQ)})^3, \label{eq:SSecond}
\end{split}
\end{equation}
and clearly the third one also satisfies the above estimate for sufficiently large $A$.
Setting $\Delta=AT^{2/3}(\log{T})^{B_1}$ for some constant $B_1$ in \eqref{eq:SFirst} and \eqref{eq:SSecond}, we obtain
  $  S(T, \Delta) \ll T^{2/3}(\log{T})^{3-B_1/2}.$
Next, we optimise the constant $B_1$.

The term $\Delta\log^{4}{T}$ in \eqref{eq:ErrorBoundary} contributes the term $T^{2/3}(\log{T})^{4+B_1}$. Thus we choose $B_1=-2/3$. Combining the estimates with the bound \eqref{eq:ErrorBoundary}, we can conclude that $|E_2(2T)-E_2(T)|\ll T^{2/3} \left(\log{T}\right)^{10/3}$. The term $T^{1/2}(\log T)^B$ is smaller, since $B$ is fixed. Summing over dyadic intervals proves the theorem.
\end{proof}

\begin{remark}
    Using the method above, future improvements for $E_2(T)$ should come improving the estimate for $S(T, \Delta).$ 
    We also note that choosing an exponent less than $1/2$ for $\log{T}$ in $T_1,\ldots, T_4$ in Lemma \ref{lemma:I4} induces an error term $O(T)$ in \eqref{eq:ErrorBound}, which would be too large. Lemma \ref{lemma:I4Boundary} avoids this restriction by retaining boundary integrals rather than shortening these shifts.
\end{remark}

\section{Proofs of Theorem \ref{die} and Corollary \ref{cor:12moment}}

In order to prove our results, we apply estimates for spectral mean values of the Maass waveform $L$-functions. However, the current\footnote{We note  that the method of proving \cite[Corollary 1]{Motohashi1992} was originally due to Iwaniec \cite{Iwaniec}.}  known estimates (see \cite[Corollary 1]{Motohashi1992}) consider sums that are too short. Extending these results trivially would produce too high a power of $\log T$ in Remark \ref{cups}. Hence, we provide the following extension of \cite[Corollary 1]{Motohashi1992}.

\begin{Lem}
\label{lemma:loglogSaving2}
Assume that $\{a_f\}$ is an arbitrary complex vector, $0<T\leq t_1<t_2<\ldots< t_F \leq 2T$, $|t_f-t_g|\geq\delta>0$ for $f\neq g$, $\log{T} \asymp \log{K}$ and $\Delta \in [K/\log{K}, K]$. Then we have
\begin{equation*}
\sum_{K \leq \kappa_j <K+\Delta} \alpha_j H_j\left(\frac{1}{2}\right)^2 \left|\sum_{f=1}^F a_f t_f^{i \kappa_j}\right|^2 \ll \left( \Delta+T\delta^{-1}\right)||a_f||^2K\log{K}.
\end{equation*}
\end{Lem}
\begin{proof}
We follow the proof of  \cite[Corollary 1]{Motohashi1992}. Namely, let us conclude that
\begin{equation}
\label{eq:estI}
\sum_{K \leq \kappa_j <K+\Delta} \alpha_j H_j\left(\frac{1}{2}\right)^2 \left|\sum_{f=1}^F a_f t_f^{i \kappa_j}\right|^2 \ll K^{-2}\sum_{f,g=1}^F a_f\overline{a_g} I(1; h_0(*, t_f/t_g)),
\end{equation}
where 
\begin{equation*}
I(1; h_0(*, t_f/t_g)):= \sum_{j=1}^\infty \alpha_j H_j\left(\frac{1}{2}\right)^2 h_0(\kappa_j)\cos(\kappa_j \log{(t_f/t_g)})
\end{equation*}
and
\begin{equation*}
h_0(r):=\left(r^2+\frac{1}{4}\right)\left(\exp\left(-\left(\frac{r-K}{\Delta}\right)^2\right)+\exp\left(-\left(\frac{r+K}{\Delta}\right)^2\right)\right).
\end{equation*}
In order to estimate the right-hand side of \eqref{eq:estI}, we want to estimate the functions 
\begin{align*}
& \Phi(m):=\int_{(\beta)} m^s\Gamma\left(\frac{1}{2}-s\right)^2 \tan(\pi s)h^*(s) \, ds, \\
& \Psi(m):=\int_{(\beta)} m^s\Gamma\left(\frac{1}{2}-s\right)^2 \cos(\pi s)^{-1}h^*(s) \, ds, \\
& h^*(s):= \int_{-\infty}^\infty rh_0(r)\cos(r \log{(t_f/t_g)}) \frac{\Gamma(s+ir)}{\Gamma(1-s+ir)} \, dr,
\end{align*}
where $-3/2<\beta<1/2$ --- see  \cite[Lemma, p.\ 266]{Motohashi1992}. 

Let us start with the function $\Phi(m)$. As we see in \cite[p. 270]{Motohashi1992}, it is sufficient to estimate 
\begin{equation}
\label{eq:IntegralPhi}
\int_{-\infty}^\infty (r-i) h_0(r-i) \tanh(\pi r)\left(\frac{\xi y (1-y)}{m+y}\right)^{i(r-i)} \, dr,
\end{equation}
where $\xi =t_f/t_g$, $y \in[0,1]$ and $m \geq 1$. We will divide the range of integration in \eqref{eq:IntegralPhi} into three parts: $|r| \geq K \log{K}$, $\log{K}\leq |r|< K \log{K}$ and $|r|<\log{K}$. The first part contributes 
\begin{equation*}
\begin{split}
&\ll m^{-1}\int_{K\log{K}}^\infty r^3\exp\left(-\left(\frac{r-K}{\Delta}\right)^2\right) \, dr \ll m^{-1}K^{4}\int_{\log{K}}^\infty v^3\exp\left(-\frac{v^2}{2}\right) \, dv\\
&\ll m^{-1}K^{4}(\log{K})^2\exp\left(-\frac{(\log{K})^2}{2}\right)\ll m^{-1}\exp\left(-\frac{1}{3}(\log{K})^2\right).
\end{split}
\end{equation*}
We have $\tanh(\pi r)=\text{sgn}(r)+O(K^{-2\pi})$ in the second case. The contribution for the integral \eqref{eq:IntegralPhi} coming from the previous $O$-term is $\ll (mK)^{-1}$. Moreover, we can estimate the contribution coming from the third case as $\ll m^{-1}(\log{K})^4$.
Similarly as in  \cite[p. 271]{Motohashi1992}, using the stationary phase method we obtain 
$
\Phi(m) \ll m^{-1}(\log{K})^4. 
$
Hence, applying $\text{(2.34)}_{\text{2}}$ in  \cite{Motohashi1992} gives an estimate $O\left((\log{K})^4T\delta^{-1}||a_f||^2\right)$ to the right-hand side of \eqref{eq:estI}.

Similarly as with the case $\Phi(m)$, we want to estimate the integral
\begin{equation}
\label{eq:PsiIntegral}
-\int_{-\infty}^\infty \frac{(r-i)h_0(r-i)}{\cosh(\pi r)}\cdot\left(\frac{\xi y(1-y)}{m-y}\right)^{i(r-i)} \, dr
\end{equation}
for $\Psi(x)$ if $m>1$. Using a series expansion for the term $1/\cosh(\pi r)$, we can conclude that $1/\cosh(\pi r) \ll e^{-\pi |r|}$, so that and \eqref{eq:PsiIntegral} can be estimated with
$$
\frac{\xi}{m}\int_0^\infty r^3 \exp(-\pi |r|) \, dr \ll m^{-1}.
$$
Hence, we have $\Psi(m) \ll m^{-1}$ for all $m>1$. With $f=1$, the sum $(2.34)_4$ in \cite{Motohashi1992} is empty and in $(2.34)_3$ we have $1+m/f \geq 2$. This means that $\Psi(1)$ occurs only in $(2.34)_5$ and we have\footnote{We note that Motohashi, in \cite[(2.43)]{Motohashi1992}, his version of our (\ref{eq:Psi1}), has `cash' instead of `$\cosh$'.}
\begin{equation}
\label{eq:Psi1}
\Psi(1)=2\pi^2 \int_{-\infty}^\infty rh_0(r) \frac{\sinh(\pi r)}{(\cosh(\pi r))^2} \, dr \ll  \int_{-\infty}^\infty \left| \frac{r^3 \sinh(\pi r)}{(\cosh(\pi r))^2} \right|\, dr \ll 1.
\end{equation}
Hence, $(2.34)_3-(2.34)_5$ in \cite{Motohashi1992} contribute $\ll 1$.  The term  $(2.34)_7$ in \cite{Motohashi1992}  is negative in our case, too, and $(2.34)_6$ is clearly $\ll 1$. Hence, putting our estimates for $(2.34)_2-(2.34)_7$ in \cite{Motohashi1992} to the right-hand side of \eqref{eq:estI} gives $\ll  T\delta^{-1} ||a_f||^2$. 

Thus, following the lines of \cite[p. 271]{Motohashi1992}, the right-hand side of \eqref{eq:estI} is $\ll$
\begin{equation*}
K\log{K}\int_{-\infty}^\infty \bigl(e^{-(r-K)^2/(2\Delta^2)}+e^{-(r+K)^2/(2\Delta^2)}\bigr)\left(\sum_{f,g=1}^F a_f\overline{a_g}\left(\frac{t_f}{t_g}\right)^{ir}\right) \, dr.
\end{equation*}
We evaluate the previous integral using substitutions $x_1=(r-K)/\Delta$ and $x_2=(r+K)/\Delta$. The diagonal terms $f=g$ contribute $\ll \Delta ||a_f||^2$, and the non-diagonal terms contribute $\ll T\delta^{-1}||a_f||^2$ since $|t_f-t_g|\geq\delta>0$. This finishes the proof.
\end{proof}

\begin{cor}
\label{cor:Integral}
Let $T>0$, let $K$ be sufficiently large, let $\{a_j\}$ be an arbitrary complex vector, and suppose that $|M_j(y)| \leq c(1+|y|)^{-2}$ for all $j$ and $y\in\mathbb R$, where $c>0$ is fixed. Then, we have
\begin{multline*}
\int_T^{2T} \left|\int_{-\infty}^\infty t^{1/2-yi}  \sum_{K\leq \kappa_j <2K} \sqrt{\alpha_j} H_j(1/2) a_j M_j(y) t^{\pm i\kappa_j} \, dy\right|^2 \, dt \\
\ll T^2 K(\log{K})\sum_{K\leq \kappa_j <2K} |a_j|^2.
\end{multline*}
\end{cor}

\begin{proof}
Use the sampling points $u_f=K+(f-1/2)K/F$, with $1\leq f\leq F$.
Define the matrix by $A_{f,j}=\sqrt{\alpha_j}H_j(1/2)u_f^{-i\kappa_j}$.
Lemma \ref{lemma:loglogSaving2}, applied at scale $K$ with separation
$K/F$ and spectral length $K$, bounds both $\|A\|_2^2$ and $\|A^*\|_2^2$
by $O((K+F)K\log K)$.
The same bound holds for the conjugate matrix. Thus, for either sign,
\begin{equation*}
 \frac1F\sum_{f=1}^F
 \left|\sum_{K\leq\kappa_j<2K}\sqrt{\alpha_j}H_j(1/2)
 b_j(u_f/K)^{\pm i\kappa_j}\right|^2
 \ll(1+K/F)K\log K\sum_j|b_j|^2.
\end{equation*}
Here the common phase $K^{\pm i\kappa_j}$ is absorbed into the coefficients.
Letting $F\to\infty$ and then putting $t=Tv$ gives
\begin{equation*}
 \int_T^{2T}\left|\sqrt t\sum_{K\leq\kappa_j<2K}
 \sqrt{\alpha_j}H_j(1/2)b_jt^{\pm i\kappa_j}\right|^2dt
 \ll T^2K\log K\sum_j|b_j|^2.
\end{equation*}
Apply this with $b_j=a_jM_j(y)$. Minkowski's inequality in $y$ and
$|M_j(y)|\ll(1+|y|)^{-2}$ prove the claim.
The initial rescaling avoids any condition $\log T\asymp\log K$.
\end{proof}

\begin{Lem}
\label{lemma:shiftedSMean}
Let $S$ be as in \eqref{eq:defS}, fix $0<\varepsilon_1<1/8$, and put
\begin{equation*}
\Delta=T^{1/2}(\log T)^{-2},\qquad
L=\left\lfloor(1/4+\varepsilon_1)\log_2T\right\rfloor.
\end{equation*}
Assume $a\in\{1,2\}$ and $l\in\{1,2,\ldots,L\}$. Then, we have
\begin{equation}\label{eq:shiftedSMean}
\int_T^{2T}|S(a2^{-l}t\pm\Delta\log(2^{-l}t),\Delta)|^2\,dt
\ll 2^{-l}T^2(\log T)^8.
\end{equation}
\end{Lem}

\begin{proof}
We set $\tau(l, t):=a2^{-l}t\pm\Delta \log{(t 2^{-l})}$. As in \eqref{eq:defS}, we have
\begin{align}
    & S\left(\tau(l,t), \Delta\right)=\left(\frac{\tau(l,t)}{2}\right)^{1/2} \pi \sum_{\kappa_j \leq T\Delta^{-1}\log{T}} \alpha_j \kappa_j^{-3/2} H_j(1/2)^3\cos{\left(\kappa_j\log{\frac{4 e \tau(l,t)}{\kappa_j}}\right)} \cdot \notag\\
    &\quad\quad \cdot \exp\left(-\left(\frac{\Delta \kappa_j}{2\tau(l,t)}\right)^2\right)  +O(1). \label{eq:sumSmean}
\end{align}
We will now consider the exponential term in the sum.

The image of $\tau(l,t)$ can be covered by a fixed number of intervals $[X(l),2X(l)]$, with $X(l)\asymp2^{-l}T$. We consider one such interval and restrict $t$ to its preimage $\mathcal T_l\subset[T,2T]$. Since $\tau'(l,t)\asymp2^{-l}$, changing variables from $t$ to $\tau$ costs a factor $O(2^l)$ in the square integral. Let $g(u)\in [0,1]$ be a bump function such that $g(u)=1$ if $u\in[1,2]$ and $g(u)=0$ if $u\leq 1/2$ or $u\geq 4$. Let us also define $g_l(u):=g(u/X(l))$ for $u>0$, $B_j:=(\Delta \kappa_j)^2/4$ and $Z_{l,B_j}(u):=g_l(u)\exp(-B_j/u^2)$. Thus, the function $Z_{l,B_j}$ is compactly supported and smooth, and hence its Mellin transform $M_{l,B_j}$ is entire and we can write 
\begin{equation}
\label{eq:defZMmean}
Z_{l,B_j}(y)=\frac{1}{2\pi i}\int_{(\omega)}M_{l,B_j}(z)y^{-z} \, dz \quad\text{and}\quad M_{l,B_j}(s)=\int_0^{\infty} Z_{l,B_j}(x)x^{s-1}\, dx
\end{equation}
for $\omega:=1/\log{X(l)}$. 
We will replace the exponential term in \eqref{eq:sumSmean} with $Z_{l,B_j}(y)$ as in \eqref{eq:defZMmean} after changing the line of integration from $\Re(z)=\omega$ to $\Re(z)=0$. 

Write $q=B_j/X(l)^2$ and $f_q(x)=g(x)e^{-q/x^2}$.
For $r=0,1,2$, the derivatives $f_q^{(r)}(x)$ are uniformly bounded on
$[1/2,4]$, independently of $q\geq0$: each term is a bounded derivative
of $g$, a bounded power of $x^{-1}$, and a factor $q^m e^{-q/x^2}$.
Two integrations by parts therefore give, uniformly for
$0\leq\sigma\leq\omega$ and $|y|\geq1$,
\begin{equation*}
 M_{l,B_j}(\sigma+iy)
 \ll X(l)^\sigma |\sigma+iy|^{-2}\ll |y|^{-2}.
\end{equation*}
For $|y|\leq1$, the defining integral gives
$|M_{l,B_j}(iy)|\leq\log8$.
Thus
\begin{equation}\label{eq:MellinDecay}
 M_{l,B_j}(iy)\ll(1+|y|)^{-2},
\end{equation}
uniformly in $l$ and $j$.
Since $M_{l,B_j}$ is entire and the horizontal integrals in
\eqref{eq:defZMmean} are $O(\omega |y|^{-2})$,
the contour can be moved to $\Re z=0$.
 
Put $U=T\Delta^{-1}\log T$ and $M=\lfloor\log_2U\rfloor$.
For $K=2^{-m}U$, $1\leq m\leq M$, let $B_K(u)$ be either exponential piece of the
 sum in \eqref{eq:defS} on $K\leq\kappa_j<2K$, including
$\sqrt u$ and the Gaussian and omitting the fixed numerical factor.
The finitely many small spectral parameters and the Gaussian tail
contribute $O(X(l)^{1/2})$, whose square integral after a change
of variables is $O(2^{-l}T^2)$.
For sufficiently large $K$, Mellin inversion in \eqref{eq:defZMmean}
and \eqref{eq:MellinDecay} allows us to apply Corollary \ref{cor:Integral}
with $M_j(y)=M_{l,B_j}(iy)$ and
\begin{equation*}
 a_j=\sqrt{\alpha_j}H_j(1/2)^2\kappa_j^{-3/2}
 \left(\frac{4e}{\kappa_j}\right)^{\pm i\kappa_j}.
\end{equation*}
By \eqref{gluckliche}, $\sum_{K\leq\kappa_j<2K}|a_j|^2
\ll K^{-1}(\log K)^6$. Thus the corollary bounds the square integral
of $B_K$ over $[X(l),2X(l)]$ by $X(l)^2(\log K)^7$.

To combine the dyadic ranges without losing an additional logarithm, we use almost-orthogonality. Let $f$ be a fixed smooth nonnegative function supported in $[1/2,3]$ and equal to one on $[1,2]$. The preceding square-integral estimate holds also with the weight $f(u/X(l))$, by covering its support with a fixed number of dyadic intervals.
The preceding bound therefore gives
\begin{equation*}\label{eq:oneBlockMean}
\int_0^\infty |B_K(u)|^2f(u/X(l))\,du\ll X(l)^2(\log(2K))^7.
\end{equation*}
After the substitution $u=X(l)e^v$, repeated integration by parts gives, for every fixed $N\geq1$,
\begin{equation*}\label{eq:crossKernel}
\begin{split}
\left|\int_0^\infty u^{1-i(\kappa_j-\kappa_m)}
e^{-\Delta^2(\kappa_j^2+\kappa_m^2)/(4u^2)}f(u/X(l))\,du\right|
\ll_N X(l)^2(1+|\kappa_j-\kappa_m|)^{-N}.
\end{split}
\end{equation*}
The bound is uniform in $\Delta$, since derivatives of the Gaussian give bounded powers of its argument times the same exponential. By the first estimate in \eqref{gluckliche},
\begin{equation*}
\sum_{K\leq\kappa_j<2K}\alpha_jH_j(1/2)^3\kappa_j^{-3/2}
\ll K^{1/2}(\log(2K))^3.
\end{equation*}
Put the dyadic ranges into four classes according to the index $m$ modulo four. Two distinct ranges $K<K'$ in one class have ratio at least $16$, so their inner product is
\begin{equation*}
\ll_N X(l)^2 K^{1/2}(K')^{1/2-N}(\log(2K))^3(\log(2K'))^3.
\end{equation*}
Taking $N=4$ and summing over all distinct pairs gives $O(X(l)^2)$. Cauchy's inequality combines the four classes and the two signs. Since
\begin{equation*}
\sum_{m=1}^{M}(\log(2^{1-m}U))^7\ll(\log T)^8,
\end{equation*}
we obtain, after the change of variables,
\begin{equation*}
\int_T^{2T}|S(\tau(l,t),\Delta)|^2\,dt
\ll 2^lX(l)^2(\log T)^8\ll2^{-l}T^2(\log T)^8.
\end{equation*}
Summing over the bounded number of preimages $\mathcal T_l$ proves the claim. The logarithm from the number of dyadic ranges has been included: direct Minkowski would instead give $(\log T)^9$.
\end{proof}

We also retain the following large-value estimate for $E_2(t)$. It gives another consequence of Lemma \ref{lemma:loglogSaving2}; the sharper mean square in Theorem \ref{die} will be obtained directly from Lemma \ref{lemma:shiftedSMean}.
\begin{proof}[Proof of Theorem \ref{die}]
It is sufficient to prove the assertion for $A=2$, the other cases following from H\"older's inequality. Fix $0<\varepsilon_1<1/8$, put
\begin{equation}\label{eq:meanDelta}
\Delta=T^{1/2}(\log T)^{-2},\qquad
L=\left\lfloor(1/4+\varepsilon_1)\log_2T\right\rfloor,
\end{equation}
so that
\begin{equation}\label{eq:terminalScale}
T^{3/4-\varepsilon_1}\leq2^{-L}T<2T^{3/4-\varepsilon_1}.
\end{equation}
For $T\leq t\leq2T$, telescope $E_2(t)$ through the points $2^{-l}t$.

We use the upper and lower local formulae \cite[(4.13), (4.14)]{Ivic14}, together with the sharper remainder estimate in \cite[(1.13)]{IvicM95}. Writing $u_l=2^{-l}t$ and $D_l(t)=E_2(2u_l)-E_2(u_l)$, these give

\begin{align}
D_l(t)&\leq S(2u_l+\Delta\log u_l,\Delta)-S(u_l-\Delta\log u_l,\Delta)+O\left(\Delta\log^{5} T+u_l^{1/2}\right),\label{eq:localUpper}\\
D_l(t)&\geq S(2u_l-\Delta\log u_l,\Delta)-S(u_l+\Delta\log u_l,\Delta)
-O\left(\Delta\log^{5} T+u_l^{1/2}\right).\label{eq:localLower}
\end{align}

Here it is important to use the sharper remainder. The coarser error in \eqref{eq:integratedI} is not sufficient for this argument.

The formulae apply uniformly at all the scales in \eqref{eq:terminalScale}. For fixed sufficiently large $D$, their lower condition follows from
\begin{equation*}
(2^{-l}T)^{1/2}(\log(2^{-l}T))^{-D}\leq\Delta,
\end{equation*}
and their upper condition follows from
\begin{equation*}
\frac{\Delta}{2^{-l}T}\ll T^{-1/4+\varepsilon_1}(\log T)^{-2}
=o\left(\exp(-\sqrt{\log(2^{-l}T)})\right).
\end{equation*}
Moreover, Theorem \ref{thm:fourthMoment} and \eqref{eq:terminalScale} show that
\begin{equation*}
|E_2(2^{-L}t)|\ll T^{1/2-2\varepsilon_1/3}(\log T)^{10/3}\ll T^{1/2}.
\end{equation*}
Summing \eqref{eq:localUpper} and \eqref{eq:localLower}, and using $L\ll\log T$ and $\sum_{l\leq L}u_l^{1/2}\ll T^{1/2}$, we obtain
\begin{equation}\label{eq:recoverE2}
|E_2(t)|\ll
\sum_{l=1}^{L}\sum_{c=1}^{2}\sum_{\pm}
|S(c2^{-l}t\pm\Delta\log(2^{-l}t),\Delta)|
+\Delta(\log T)^6+T^{1/2}.
\end{equation}
In particular, the logarithm from summing the smoothing errors over $l$ is retained.

Lemma \ref{lemma:shiftedSMean} applies to each of the four shifted terms in \eqref{eq:recoverE2}, with the same $\Delta$ and $L$. In particular, its factor $2^{-l}$ includes the Jacobian from changing variables.

Minkowski's inequality in \eqref{eq:recoverE2}, followed by \eqref{eq:meanDelta} and \eqref{eq:shiftedSMean}, yields
\begin{align*}
\|E_2\|_{L^2[T,2T]}
&\ll T(\log T)^4\sum_{l=1}^{L}2^{-l/2}
+T^{1/2}\left(\Delta(\log T)^6+T^{1/2}\right)\\
&\ll T(\log T)^4.
\end{align*}
Hence $\int_T^{2T}E_2(t)^2\,dt\ll T^2(\log T)^8$. Summing over dyadic intervals proves the case $A=2$. The case $A=0$ is immediate. For $0<A<2$, H\"older's inequality gives
\begin{equation*}
\int_0^T|E_2(t)|^A\,dt
\leq T^{1-A/2}\left(\int_0^TE_2(t)^2\,dt\right)^{A/2}
\ll_A T^{1+A/2}(\log T)^{4A}.
\end{equation*}
\end{proof}

\begin{remark}
For fixed $A\geq2$, Theorems \ref{thm:fourthMoment} and \ref{die} also give
\begin{equation*}
\int_0^T|E_2(t)|^A\,dt
\ll_A T^{(2A+2)/3}(\log T)^{(10A+4)/3},
\end{equation*}
by bounding $|E_2(t)|^{A-2}$ pointwise and dealing with a bounded initial interval separately.
\end{remark}

\begin{proof}[Proof of Corollary \ref{cor:12moment}]
We first prove the estimate with logarithmic exponent $5k/4$, and then obtain the stated bound by interpolation. We work on $[T,2T]$. First, we notice that if $|\zeta(1/2+it)| \ll T^{\frac{k-4}{8k}}(\log{T})^{\frac{a(k-4)+k+16}{4k}}$, where $a=4$, then this first estimate holds. Hence, it is sufficient to consider cases where this is not true.
Let us denote 
$$
V_n:=2^nT^{\frac{k-4}{8k}}(\log{T})^{\frac{a(k-4)+k+16}{4k}}
$$ 
for $n=0,1,\ldots, O(\log{T})$, $I_j:=[(j-1)/3,j/3]$, where $j\in \mathbb{Z}$, $\hat{I}_j:=I_{j-1}\cup I_j \cup I_{j+1}$,
\begin{align*}
& H(V_n):=\{I_j: I_j\cap[T,2T]\neq\emptyset, \max_{t \in I_j} |\zeta(1/2+it)| \in (V_n, 2V_n]\}, \\
& J(V_n):=\bigcup_{I_j \in H(V_n)} \hat{I}_j \quad\text{and}\quad S(V_n):=\sum_{I_j \in H(V_n)} \int_{I_j}|\zeta(1/2+it)|^4 \, dt.
\end{align*}
 For each $n$, we have
\begin{equation*}
\label{eq:SVn}
\begin{split}
\sum_{I_j\in H(V_n)}\int_{I_j}|\zeta(1/2+it)|^k\,dt
&\leq(2V_n)^{k-4}\sum_{I_j\in H(V_n)}\int_{I_j}|\zeta(1/2+it)|^4\,dt\\
&=(2V_n)^{k-4}S(V_n)\ll V_n^{k-4}S(V_n).
\end{split}
\end{equation*}
The aim is to bound $S (V_n)$.

By the definition of $H(V_n)$, for each $I_j \in H(V_n)$ there is at least one $t_j \in I_j$ such that $|\zeta(1/2+it_j)|\geq V_n$. Thus, by \cite[Remark 2]{Balu2} we have
\begin{equation}
\label{eq:Vn4log}
\begin{split}
V_n^4 \ll |\zeta(1/2+it_j)|^4 &\ll (\log{T})\left(1+\int_{t_j-1/3}^{t_j+1/3} |\zeta(1/2+iu)|^4 \, du\right)\\
& \ll (\log{T})\int_{\hat{I}_j} |\zeta(1/2+iu)|^4 \, du.
\end{split}
\end{equation}
The additive constant in \eqref{eq:Vn4log} is absorbed since $V_0^4/\log T\to\infty$. Summing over all $I_j\in H(V_n)$ and using bounded overlap, we obtain
\begin{equation}
\label{eq:SstarVn}
 |H(V_n)|\ll V_n^{-4}(\log{T})\int_{J(V_n)} |\zeta(1/2+iu)|^4 \, du=:V_n^{-4}(\log{T})S^*(V_n).
\end{equation}

Let us now cover $J(V_n)$ with the cells of a regular mesh of length $\Delta_n:=cV_n^4\log^{-5}T$ which meet $J(V_n)$, where $c>0$ is a fixed constant that is specified later. Let $\tau_1<\tau_2<\ldots<\tau_{M_n}$ be their starting points. Any two cells have at most one point in common, and the enlarged intervals $[\tau_r-2\Delta_n,\tau_r+2\Delta_n]$ have bounded overlap. By \cite[Theorem 5]{Bourgain}, $V_n\ll T^{13/84+\varepsilon}$ for each nonempty level. Taking $\varepsilon>0$ sufficiently small gives $1\ll\Delta_n\ll T^{2/3}$, so the enlarged intervals lie in $[T/2,3T]$ for sufficiently large $T$. Each $\hat I_j$ meets at most three cells, and hence $M_n\ll|H(V_n)|$.
Hence, \cite[(4.4)]{IM94} gets the form
\begin{equation}
\label{eq:replacement44}
\sum_{r=1}^{M_n}\int_{\tau_r-2\Delta_n}^{\tau_r+2\Delta_n} \left|E_2(t)\right| \, dt \ll \left(\Delta_nM_n\right)^{\frac{k-8}{k-4}}\left(\int_{T/2}^{3T}\left|E_2(t)\right|^{(k-4)/4} \, dt\right)^{4/(k-4)}, 
\end{equation}
and by \cite[(4.3)]{IM94}, \eqref{eq:replacement44} and Theorem \ref{die} we obtain
\begin{equation*}
\sum_{r=1}^{M_n}\int_{\tau_r}^{\tau_r+\Delta_n} \left|\zeta\left(\frac{1}{2}+it\right)\right|^{4} \, dt \ll M_n\Delta_n\log^4{T}+M_n^{\frac{k-8}{k-4}}\Delta_n^{-\frac{4}{k-4}}T^{\frac{k+4}{2(k-4)}}(\log{T})^{a}.
\end{equation*}
 We have $M_n \ll |H(V_n)|$ and thus by \eqref{eq:SstarVn} we obtain
\begin{equation*}
S^*(V_n) \ll cS^*(V_n)+c^{-\frac{4}{k-4}}(S^*(V_n))^{\frac{k-8}{k-4}}V_n^{\frac{16-4k}{k-4}}T^{\frac{k+4}{2(k-4)}}(\log{T})^{a+\frac{k+12}{k-4}}.
\end{equation*}
Now, we choose $c$ to be small enough, independently of $n$, so that we obtain
$$
S^*(V_n)\ll V_n^{-(k-4)}T^{\frac{k+4}{8}}(\log{T})^{\frac{a(k-4)+k+12}{4}}.
$$
Since $S(V_n)\ll S^*(V_n)$ and there are at most $O(\log T)$ different values of $n$, we obtain
\begin{equation*}
\int_T^{2T}|\zeta(1/2+it)|^k\,dt
\ll_k T^{1+(k-4)/8}(\log T)^{5k/4}.
\end{equation*}
For $k=12$, this gives $T^2(\log T)^{15}$: the bound for each level has fourteen logarithms, and the number of levels supplies the fifteenth. More generally, the same argument for $k=12$, with $a=C_2/2$ in \eqref{laune}, gives $T^2(\log T)^{C_2+7}$. Summing over dyadic intervals and applying H\"older's inequality with the fourth moment now gives
\begin{align*}
\int_0^T|\zeta(1/2+it)|^k\,dt
&\leq\left(\int_0^T|\zeta(1/2+it)|^4\,dt\right)^{(12-k)/8}\\
&\quad\times\left(\int_0^T|\zeta(1/2+it)|^{12}\,dt\right)^{(k-4)/8}\\
&\ll_k T^{1+(k-4)/8}(\log T)^{(11k-12)/8}.
\end{align*}
This proves the corollary.
\end{proof}

\section{Discussion, extension, and future work}
We note that in \cite[Equation (4.5)]{IM94}, there is the estimate
    \begin{equation*}
        \left|\zeta(1/2+it)\right|^{2k} \ll (\log{t})\left(1+\int_{t-1/3}^{t+1/3} \left|\zeta(1/2+iu)\right|^{2k} \, du \right).
    \end{equation*}
    It is not clear to us that the $\log{t}$ term, which comes from growth estimates of $\zeta(1/2 + it)$, can be reduced.

We note that Heath-Brown \cite{RHBMean} used the 12th power moment to prove $$\int_{0}^{T} |\zeta(5/8 + it)|^{8}\, dt \ll T (\log T)^{38},$$ and in \cite{RHBZero} to improve on certain zero-density results. We expect that these results could be improved very slightly using our work. Similarly, Graham \cite{Graham} (see also \cite[p.\ 179]{Titchmarsh}) adapted Heath-Brown's method to prove $$\int_{0}^{T} |\zeta(5/7 + it)|^{196}\, dt \ll T^{14} (\log T)^{425}.$$ We note (very briefly) the 12th power moment for Dirichlet $L$-functions, also with logarithmic powers, in Motohashi \cite{MotohashiV3}, improving work by Meurman \cite{Meurman84}.

While we have not pursued potential improvements to these results, we do list one final result indirectly related to this work. 
Let $a(n)$ denote the number of non-isomorphic Abelian groups of order $n$, and let $A(x) = \sum_{n\leq x} a(n)$. The problem is to estimate $\Delta(x)$, the error term in the approximation $A(x) = \sum_{k=1}^{5}c_{k}x^{1/k} + \Delta(x)$.

Ivi\'{c} \cite{IvicOmega} proved $\int_{1}^{X} \Delta(x)^{2} \, dx = \Omega(X^{4/3} \log X)$ --- see also \cite{Balu}. Heath-Brown \cite{Aster} gave the following upper bound that is sharp up to logarithmic powers
\begin{equation}\label{barge}
\int_{1}^{X} \Delta(x)^{2} \, dx \ll X^{4/3} (\log X)^{89}.
\end{equation}
An ingredient in this proof is the 12th power moment. Since Heath-Brown had `made no attempt to obtain a good exponent for the power of $\log X$' in (\ref{barge}), we had hoped that a careful analysis of the role of the 12th power moment could reduce the power of 89 in (\ref{barge}).

Heath-Brown's proof requires careful estimation of mean values of various Dirichlet polynomials. This is done by splitting the ranges of summation into four cases and by frequently using the following fact. If $A_{i}, \alpha_{i} \geq 0$ and $\sum \alpha_{i} = 1$, then $\min\{A_{1}\, \ldots, A_{k}\} \leq A_{1}^{\alpha_{1}}\cdots A_{k}^{\alpha_{k}}$. Where powers of $\log X$ are concerned, Heath-Brown bounds the minimum of the $A_{i}$'s by the \textit{maximum} of the $A_{i}$'s. By retaining the exponents at each step we make a  cheap improvement. So, for example, the exponent of 65 in \cite[(3.1)]{Aster} can be reduced to $77/2$. This actually comes from Heath-Brown's `Case 1' (the only case not to invoke the 12th power moment), and is the largest of the four estimates. A quick examination of the proof shows that
\begin{equation*}\label{barge2}
\int_{1}^{X} \Delta(x)^{2} \, dx \ll X^{4/3} (\log X)^{62.5}.
\end{equation*}

 \end{document}